\documentclass[a4paper,12pt]{amsart}
\usepackage{a4wide}
\usepackage{amssymb,amsmath, amsthm}
\usepackage[all]{xy}
\usepackage[OT2,OT1]{fontenc}
\usepackage{graphicx}
\usepackage{pst-plot}
\usepackage{comment}
\usepackage[justification=justified,labelfont={small,sf},font={small,sf},
aboveskip=0em,belowskip=1em]{caption}
\usepackage{epsfig}
\usepackage{fancyhdr}
\usepackage{enumerate}
\usepackage{url}
\usepackage{hyperref}
\newcommand\cyr{%
 \renewcommand\rmdefault{cmr}%
 \renewcommand\sfdefault{wncyss}%
 \renewcommand\encodingdefault{OT2}%
\normalfont\selectfont}
\DeclareTextFontCommand{\textcyr}{\cyr}

\newtheorem{thm}{Theorem}[section]

\newtheorem{prop}[thm]{Proposition}
\newtheorem{lemma}[thm]{Lemma}

\newtheorem{conj}[thm]{Conjecture}
\theoremstyle{definition}

\theoremstyle{remark}
\newtheorem{rk}[thm]{Remark}

\newcommand\into{\hookrightarrow}
\newcommand\Q{\mathbb{Q}}

\newcommand\A{\mathbb{A}}
\newcommand\Z{\mathbb{Z}}

\newcommand\Char{\mathop{\rm char}\nolimits}

\newcommand{\disc}{\operatorname{disc}}

\newcommand{\Pic}{\operatorname{Pic}}

\newcommand{\To}{\longrightarrow}
\newcommand{\BP}{{\mathbb P}}

\begin{document}
\title[Explicit Kummer varieties]{Explicit Kummer varieties of hyperelliptic Jacobian threefolds}

\author{J.\thinspace{}Steffen M\"uller}
\address{Institut f\"ur Mathematik,
         Carl von Ossietzky Universit\"at Oldenburg,
         26111~Oldenburg, Germany}
\email{jan.steffen.mueller@uni-oldenburg.de }

\date{\today}

\maketitle

\begin{abstract}
We explicitly construct the Kummer variety associated to the Jacobian of a hyperelliptic
curve of genus~3 that is defined over a field of
characteristic not equal to~2 and has a rational Weierstrass point defined over the
same field.
We also construct homogeneous quartic polynomials on the Kummer variety and show
that they represent the duplication map using results of Stoll. 
\end{abstract}
\section{Introduction}\label{intro}
Let $A$ be an abelian variety of dimension $g\ge 1$, defined over a field $k$.
The quotient of $A$ by the map which takes a point on $A$ to its inverse is a singular projective
variety which can also be defined over $k$ and which can be embedded into
$\BP^{2^g-1}$.
It is called the {\em Kummer variety} $K$ associated to $A$.
The complex case is discussed in~\cite[\S4.8]{CAV}.

If $A$ is an elliptic curve, then $K$ is simply the projective line over $k$.
If $A$ is the Jacobian of a genus~2 curve $C$, then this construction yields the classical
singular Kummer surface, which can be embedded as a projective hypersurface into $\BP^3$.
In the case where $\Char(k) \ne 2$ and $C$ is given by an equation of the form $y^2=f(x)$, 
an embedding and defining equation of the Kummer surface have been constructed by
Flynn~\cite{Flynnb}, see also the exposition in Chapter~3 of the book~\cite{CasselsFlynn} by Cassels-Flynn.

A particularly useful feature of the Kummer surface $K$ is that parts of the group structure on $A$ remain meaningful on $K$. 
For instance, translation by a~2-torsion point and multiplication by an integer~$n$ on $A$
commute with negation and hence descend to well-defined maps on $K$.
Moreover, one can define a pseudoaddition on $K$ using certain biquadratic forms $B_{ij}$.
Formulas for duplication, translation by a~2-torsion point and pseudoaddition have 
been found by Flynn~\cite{Flynnb} and can be downloaded
from~\url{http://people.maths.ox.ac.uk/flynn/genus2/kummer/}.
Analogues of these for the case $\Char(k)=2$ have been found by Duquesne~\cite{SD} and,
independently, a unified treatment for arbitrary $k$ and arbitrary defining equations
$y^2+h(x)y=f(x)$ of $C$ has been presented by the author in~\cite{Kchar2}.

In the present paper we discuss analogues of some of these objects for Jacobians of hyperelliptic curves
of genus~3 with a $k$-rational Weierstrass point.
In this case an embedding of the Kummer variety into $\BP^7$ has been constructed by
Stubbs~\cite{Stubbs} and we recall his construction in Section~\ref{sec:embedding}.
When $\Char(k) \notin \{2,3,5\}$, we find a complete set of defining equations for the image of $K$ under this embedding in
Section~\ref{sec:eqns}, at least for generic $C$; it turns out that in this case $K$ can be defined as the intersection of one 
quadric and~34 quartics in $\BP^7$. The genericity assumption is not hard to check for a
given curve and we conjecture that it is always satisfied.

In Section~\ref{sec:remnants} we discuss traces of the group structure on the Jacobian
that can be exhibited on the Kummer variety.
Duquesne has constructed a matrix $W_T$ representing translation by a~2-torsion point
$T\in A$ on $K$.
It turns out that under a suitable genericity assumption, biquadratic forms $B_{ij}$ as in genus~2 cannot exist in our situation,
see Proposition~\ref{nobij}.
However, building on work of Duquesne, we construct homogeneous quartic polynomials
$\delta_1, \ldots, \delta_8\in k[x_1,\ldots,x_8]$ which we conjectured to 
represent duplication on the Kummer variety in an earlier version of this article (and
in~\cite{MuellerThesis}). 
This conjecture can now be proved using recent results of
Stoll, see Theorem~\ref{deltaconj}.

Stoll has recently announced the construction of a different embedding $\xi$ of 
$K$ which is valid for arbitrary hyperelliptic genus~3 curves~\cite{StollKummer} over
fields of characteristic~$\ne 2$, see also
the talk slides \url{http://www.mathe2.uni-bayreuth.de/stoll/talks/Luminy2012.pdf}.
His embedding uses $\Pic^4(C)$, which is canonically isomorphic to $\Pic^0(C)$.
This is analogous to the description of the dual Kummer surface associated to a
genus~2 curve $C$ in terms of $\Pic^3(C)$ due to Cassels and Flynn~\cite[\S4]{CasselsFlynn}
(in dimension~3, the Kummer variety is self-dual).
Moreover, Stoll has found defining equations for his embedding of the Kummer variety, and,
using representation-theoretic arguments,
polynomials representing duplication and forms that can be used for pseudoaddition on $K$.

The formulas described in this paper are mostly too long to be reproduced here. 
They can be obtained at~\cite{HP}.
\subsection{Applications}\label{sec:apps}
In genus~2, the Kummer surface has several arithmetic applications. 
The first application is an addition algorithm on $A$ that uses pseudoaddition on $K$, see
\cite{FlynnSmart}.
Suppose that $k$ is a number field.
Letting $h$ denote the naive height on $\BP^3_k$, we get an induced naive height on the Jacobian
which can be used to search for points in $A(k)$ of bounded height.
Stoll's program~{\tt j-points} which uses this approach 
is available from \url{http://www.mathe2.uni-bayreuth.de/stoll/programs/index.html}.
Furthermore, this naive height can be used to
define and compute a canonical height $\hat{h}$ on $A$, which has numerous applications. 
See Flynn-Smart~\cite{FlynnSmart} for the construction and a first algorithm for the computation of
$\hat{h}$.
Several refinements are presented by Stoll in~\cite{StollH1} and~\cite{StollH2} and by the
author in his thesis~\cite[Chapter~3]{MuellerThesis}; see also the forthcoming
paper~\cite{MuellerStoll}.

In the genus~3 situation, we do not have an explicit description of pseudoaddition and hence we do
not get a similar addition algorithm on $A$ using the results of the present paper (such
an algorithm is given by Stoll in~\cite{StollKummer}).
We do get a height function on $K$ by restriction of the standard
height function on $\BP^7$ and an an induced naive height $h(P)=h(\kappa(P))$ on the
Jacobian, where $\kappa:A\to K\into \BP^7$ is discussed in Section~\ref{sec:embedding}. 
Using the defining equations of $K$ presented in Section~\ref{sec:eqns}, we get an 
algorithm that lists all $k$-rational points on $K$ or on $A$ up to a given
height bound.
We can also define a canonical height function
\[
\hat{h}(P)=\lim_{n\to\infty} 4^{-n} h(\kappa(2^nP)).
\]
Using Theorem~\ref{deltaconj} one can prove analogs of several results from~\cite{StollH2} and
obtain an algorithm for the computation of $\hat{h}$ as in~\cite{StollH2}.
Further details are given in $\S4.4$ of~\cite{MuellerThesis}.

An explicit theory of heights on Jacobians of hyperelliptic curves of genus~3 has recently been developed by Stoll in~\cite{StollKummer}
using his embedding $\xi$.
Algorithms for the computation of the canonical height on Jacobians of hyperelliptic curves of
any genus have been introduced by Holmes~\cite{holmes:height1} and the
author~\cite{mueller:computing}.
However, these are not easily related to a naive height suitable for point searching, as
is required by standard algorithms for saturation of finite index subgroups of the
Mordell-Weil group, such as in~\cite{StollH2}. 
A solution to this problem has recently been proposed by Holmes~\cite{holmes:height2}.

\subsection*{Acknowledgements}
This work grew out of Chapter~4 of my PhD thesis~\cite{MuellerThesis} at the University of Bayreuth.
I would like to thank my supervisor Michael Stoll for his constant help and encouragement
and for sharing a preliminary version of~\cite{StollKummer} with me.
I would also like to thank Sylvain Duquesne, Victor Flynn, Damiano Testa and Tzanko Matev for
helpful conversations and the referee for useful remarks.
Part of this work was done while I was visiting the Universit\'e Rennes~I and the
University of Oxford and I thank both institutions for their hospitality.
Finally, I would like to acknowledge support from DFG through DFG grants~STO~299/5-1 and~KU~2359/2-1.

\section{Embedding the Kummer variety}\label{sec:embedding}
In his PhD thesis, Stubbs~\cite{Stubbs} has found an explicit embedding of the Kummer variety associated to
the Jacobian of a hyperelliptic curve of genus~3 with a rational Weierstrass point into
$\BP^{7}$.
In this section we recall this embedding, also providing formulas for the image on $K$ of
non-generic points on the Jacobian.

We first fix some notation that we will use throughout this paper.
Let $k$ denote a field of characteristic $\Char(k)\ne 2$. 
We consider a hyperelliptic genus~3 curve $C$ over $k$, given by an equation
\begin{equation}\label{ceq}
  Y^2  = F(X,Z),
\end{equation}
in the weighted projective plane over $k$ with respective weights 1, 4 and~1 assigned to the
variables $X$, $Y$ and~$Z$, where
\[
F(X,Z)=f_0Z^8+f_1XZ^7+f_2X^2Z^6+f_3X^3Z^5+f_4X^4Z^4+f_5X^5Z^3+f_6X^6Z^2+f_7X^7Z
\]
is a binary octic form in $k[X,Z]$ without multiple factors such that $\deg_X(F(X,Z))=7$.
Then there is a unique point $\infty\in C$ whose $Z$-coordinate is~0.
Every hyperelliptic genus~3 curve over $k$ with a $k$-rational Weierstrass point
has an equation of the form~\eqref{ceq} over $k$.
We let $A$ denote the Jacobian of $C$ and we let $K$ denote its Kummer variety.

Every point $P\in A$ has a representative of the form
\begin{equation}\label{3rep}
 (P_1)+(P_2)+(P_3)-3(\infty),
\end{equation}
 where $P_1,P_2,P_3\in C$, and this representation is unique unless two of the $P_i$ are
 swapped by the hyperelliptic involution.
We call a point $P\in A$ {\em generic}
if  $P$ can be represented by an unordered triple of points
$(x_1,y_1,1),(x_2,y_2,1),(x_3,y_3,1)\in C$ such that all $x_i$ are pairwise distinct. 

Let $\Theta$ denote the theta-divisor on $A$ with respect to the point $\infty$.
It is well-known that $\Theta$ is ample (cf.~\cite{mumford:equations1}) and
that $2\Theta$ is base point free (cf.~\cite[\S II.6]{mumford:av}).
Hence a basis of $\mathcal{L}(2\Theta)$ gives an embedding of $K$. 
Note that $\mathcal{L}(2\Theta)$ is equivalent to a certain space of symmetric functions on
$C^3$ with restrictions on the poles as in~\cite{Flynna} or~\cite{Stubbs}.
Using this approach, Stubbs~\cite[Chapter~3]{Stubbs} has found the following basis $\kappa_1,\ldots,\kappa_8$ of the space $\mathcal{L}(2\Theta)$:
\begin{eqnarray*}
\kappa_1&=&1,\\
\kappa_2&=&x_1+x_2+x_3,\\
\kappa_3&=&x_1x_2+x_1x_3+x_2x_3,\\
\kappa_4&=&x_1x_2x_3,\\
\kappa_5&=&b_0^2-f_7\kappa_2^3+f_7\kappa_3\kappa_2-f_6\kappa_2^2+3f_7\kappa_4+2f_6\kappa_3,\\
\kappa_6&=&\kappa_2b_0^2+2b_0b_1-f_7\kappa_2^4+3f_7\kappa_3\kappa_2^2-f_6\kappa_2^3-f_7\kappa_3^2-f_7\kappa_4\kappa_2+2f_6\kappa_3\kappa_2-f_5\kappa_2^2\\
        &&+2f_5\kappa_3,\\
\kappa_7&=&b_1^2-\kappa_3b_0^2+f_7\kappa_3\kappa_2^3-2f_7\kappa_3^2\kappa_2+f_6\kappa_3\kappa_2^2+f_7\kappa_4\kappa_3-f_6\kappa_3^2+f_5\kappa_3\kappa_2-3f_5\kappa_4,\\
\kappa_8&=&\kappa_2b_1^2+2\kappa_3b_0b_1+\kappa_4b_0^2+f_7\kappa_3^2\kappa_2^2-f_7\kappa_2^3\kappa_4+f_7\kappa_2\kappa_3\kappa_4-f_7\kappa_3^3+f_6\kappa_3^2\kappa_2\\&&-f_6\kappa_4\kappa_2^2+f_5\kappa_3^2-f_5\kappa_4\kappa_2,
\end{eqnarray*}
where
\begin{eqnarray*}
b_0&=&(x_1y_2-x_2y_1-x_3y_2+x_3y_1-x_1y_3+x_2y_3)/d,\\
b_1&=&(x_3^2y_2-x_3^2y_1+x_2^2y_1+y_3x_1^2-y_2x_1^2-y_3x_2^2)/d,\\
d&=&(x_1-x_2)(x_1-x_3)(x_2-x_3).
\end{eqnarray*}
The map $\kappa:A\to \BP^7$
\[
    \kappa(P)=(\kappa_1(P),\ldots,\kappa_8(P))
\]
defines an embedding of the Kummer variety into $\BP^7$.
We also provide formulas for the values of $\kappa(P)$ when $P$ is not generic, since no
such formulas have been given by Stubbs.
Following Mumford, we can represent any $P\in A(k)$ as a pair of homogeneous forms
\[(A(X,Z),B(X,Z)),\]
where $A,\,B \in k[X,Z]$ have homogeneous degree~4 and~2, respectively.
If $P$ is generic, then 
$b_0$ and $b_1$ as defined above are simply the constant and linear coefficient of $B(X,1)\in
k[X]$.

 Suppose that $P$ is generic, satisfying $x_1x_2x_3\ne0$, and write the $\kappa_i(P)$ in
 terms of $z_j=1/x_j$ and $w_j=y_j/x_j$, $j\in\{1,2,3\}$. 
 We then multiply by the common denominator and set $w_3=0$. 
 This leads to the following formulas for $P$ having a unique representative of the form
$((x_1,y_1,1))+((x_2,y_2,1))-2(\infty)
$
and satisfying $x_1\ne x_2$:
\begin{eqnarray*}
\kappa_1(P)&=&0,\\
\kappa_2(P)&=&1,\\
\kappa_3(P)&=&x_1+x_2,\\
\kappa_4(P)&=&x_1x_2,\\
\kappa_5(P)&=&f_5+2f_6\kappa_3(P)+f_7\kappa_3(P)^2+2\kappa_4(P)f_7,\\
\kappa_6(P)&=&f_4+f_5\kappa_3(P)-f_7\kappa_4(P)\kappa_3(P),\\
\kappa_7(P)&=&-f_4\kappa_3(P)-3f_5\kappa_4(P)+f_7\kappa_4(P)^2,\\
\kappa_8(P)&=&\left(f_3\kappa_3(P)^3+f_1\kappa_3(P)+f_2\kappa_3(P)^2+2f_0-2y_1y_2+f_4\kappa_4(P)\kappa_3(P)^2-3f_3\kappa_4(P)\kappa_3(P)\right.\\&&\left.-2f_2\kappa_4(P)
+f_5\kappa_4(P)^2\kappa_3(P)-2f_4\kappa_4(P)^2+f_7\kappa_4(P)^3\kappa_3(P)+2f_6\kappa_4(P)^3\right)/(x_1-x_2)^2. 	
\end{eqnarray*}
For the case  $x_1=x_2$ it suffices to use the same $\kappa_1,\ldots,\kappa_7$ as above and 
\[
\kappa_8(P)=b_1^2+(\kappa_4(P)-\kappa_3(P)^2)(-2f_7\kappa_4(P)\kappa_3(P)-f_6\kappa_4(P)+f_7\kappa_3(P)^3+f_6\kappa_3(P)^2+f_5\kappa_3(P)+f_4),
\]
where $b_1$ is the linear coefficient of $B(X,1)\in
k[X]$ if the Mumford representation of $P$ is $(A,B)$.
 
Now consider points represented by 
\[
((x_1,y_1))-(\infty).
\]
We first look at quotients of the form $\kappa_i(P)/\kappa_5(P)$, where $P$ is again
assumed generic, and then take the limit $(x_2,y_2,1)\to (x_3,-y_3,1).$The result is 
\[
\kappa(P)=(0,0,0,0,1,-x_1,x_1^2,x_1^3).
\]
A similar argument shows that we have
\[
\kappa(O)=(0,0,0,0,0,0,0,1),
\]
where $O\in A$ is the identity element.

If $P\in A$, then we say that $x=(x_1,\ldots,x_8)\in \A^8$ is a \emph{set of Kummer coordinates
for $P$} if $\kappa(P)=(x_1:\ldots:x_8)$. 
We set
\[
K_{\A}:=\{(x_1,\ldots,x_8)\in \A^8:\;\exists \,Q\in K\text{ such that }Q=(x_1:\ldots:x_8)\}.
\]
\begin{rk}
In the general case $\deg_X F=8$ Stubbs constructs functions analogous to the
functions $\kappa_i$. 
However, these do not give an embedding of the Kummer variety, since not all points on $A$
can be represented by unordered triples of points on $C$.
See~\cite[\S3.8]{Stubbs} for a discussion.
\end{rk}
\begin{rk}\label{StollTrans}
  As discussed in the introduction, Stoll~\cite{StollKummer} has recently constructed an embedding $\xi =
(\xi_1,\xi_2,\xi_3,\xi_4,\xi_5,\xi_6,\xi_7,\xi_8)$ of the Kummer
variety into $\BP^7$  which is valid for arbitrary hyperelliptic curves $C:Y^2=F(X,Z)$ of genus~3
defined over a field of characteristic~$\ne 2$.
When $\deg_X F =7$, his embedding is related to Stubbs' embedding $\kappa$ as follows:
\[
\begin{array}{ccccccccccccccccc}
  \xi_1& = &\kappa_1 & & & & & & & & & & & & & & \\
  \xi_2& = & & &-f_7\kappa_2 & & & & & & & & & & & & \\
  \xi_3& = & & & & &f_7\kappa_3 & & & & & & & & & & \\
  \xi_4& = & & & & & & &-f_7\kappa_4 & & & & & & & & \\
  \xi_5& = &f_4\kappa_1 &+ &f_5\kappa_2 &+ &2f_6\kappa_3 &+ &3f_7\kappa_4 & -& \kappa_5& & & & & & \\
  \xi_6& = &f_3\kappa_1 &+ & f_4\kappa_2 &+ &f_5\kappa_3 & & & & & -& \kappa_6& & & & \\
  \xi_7& = &f_2\kappa_1 & & &- &f_4\kappa_3 &- &3f_5\kappa_4 & & & & & -& \kappa_7& & \\
  \xi_8& = & && -f_2f_7\kappa_2 &- &f_4f_7\kappa_3 &- &f_4f_7\kappa_4 & & & & & & &+ & f_7\kappa_8\\
\end{array}
\]
\end{rk}

\section{Defining equations for the Kummer variety}\label{sec:eqns}
In this section we compute defining equations for the Kummer variety $K$, embedded into
$\BP^7$ as in the previous section.

The following result seems to be well-known to experts in algebraic geometry, but no proof seems to
exist in the literature.
The proof given here was suggested by Tzanko Matev. 
\begin{prop}\label{KumQuart}
Let $A$ be a Jacobian variety of dimension $g\ge2$ defined over an arbitrary field and let $\Theta$ be a theta-divisor on $A$.
If $\kappa_1,\ldots,\kappa_{2^g}$ is a basis for $\mathcal{L}(2\Theta)$ and if
$\kappa=(\kappa_1,\ldots,\kappa_{2^g}):A\to\BP^{2^g-1}$, then the image $\kappa(A)$ can 
be described as an intersection of quartics.
\end{prop}
\begin{proof}
Let $\mathcal{Q}=\{q_1,\ldots,q_m\}$ denote the set of monic quadratic monomials in the
$\kappa_i$, where  $m=$${2^g+1}\choose 2$ and we assume, without loss of generality, that
$\{q_1,\ldots,q_d\}$ is linearly independent in the space $\Q(f_0,\ldots,f_7)[q_1,\ldots,q_m]$,
where $d\le m$ is the dimension of the space generated by the elements of $\mathcal{Q}$.

Let $\iota$ denote the 2-uple embedding of $\BP^{2^g-1}$ into $\BP^{m-1}$ such that for $P\in A$ we have 
\[
\iota_i(\kappa(P))=q_i(P)\quad\text{for all }i\in\{1,\ldots,m\}.
\]
Then there are $m-d$ linear relations on the image of $K=\kappa(A)$ under $\iota$.
Now consider an embedding $\beta:A\into\BP^{4^g-1}$ given by a basis of
$\mathcal{L}(4\Theta)$ whose first $d$ elements are equal to $q_1,\ldots,q_d$. Then we have a commutative diagram
\[
\xymatrix{A\ar[d]_{\kappa}\ar@{^(->}[r]^{\beta} &\BP^{4^g-1}\ar@{.>}[d]_{\gamma}\\
\BP^{2^g-1} \ar@{^{(}->}[r]_{\iota}&\BP^{m-1}&}
\]
where $\gamma$ is a rational map defined as follows: If $z=(z_1,\ldots,z_{4^g})$, then
$\gamma(z)=y$, where $y_i=z_i$ for $i=1,\dots,d$ and the other $y_i$ are determined by the
linear relations on $\mathcal{Q}$. 
By construction, we have that $\beta(A)$ is contained in the domain of $\gamma$ and in fact
\[
\gamma(\beta(A))\cong\iota(\kappa(A)).
\]
But it follows from the corollary on page~349 of~\cite{mumford:equations1} that the image of $A$ under $\beta$ is
defined by an intersection of quadrics, which then must hold for $\gamma(\beta(A))$ as
well, since $\gamma$ is defined by homogeneous polynomials of degree~1. 
As the pullback under $\iota$ of $\gamma(\beta(A))$ is isomorphic to $K$, the result follows.
\end{proof}

By Proposition~\ref{KumQuart} it suffices to find a basis for the space of quartic relations on $K$ to describe $K$. We first compute a lower bound on the dimension of this space.
For $n\ge1$ let $m(n)$ denote the number of monic monomials of degree $n$ in $\kappa_1,\ldots,\kappa_{2^g}$ and let $d(n)$ denote the dimension of the space spanned by them. Then we have $m(n)=$${2^g+n-1}\choose n$. Moreover, let $e(n)$ denote the dimension of the space of even functions in $\mathcal{L}(2n\Theta)$. By~\cite[Corollary~4.7.7]{CAV} this is equal to $(2n)^g/2+2^{g-1}$. Since a monomial of degree $n$ in the $\kappa_i$ induces an even function in $\mathcal{L}(2n\Theta)$, we always have $d(n)\le e(n)$.

In genus~2, the dimension count is given in Table~\ref{g2dims}. 
\begin{center} \begin{table}[h!]
\begin{tabular}{|c|c|c|c|}
\hline
 $n$ &$m(n)$ & $e(n)$ & $d(n)$  \\
\hline
1 &4&4&4 \\
2 &10&10&10 \\
3 &20&20&20 \\
4 &35&34&34 \\
\hline
\end{tabular}
\caption{Dimensions in genus~2}\label{g2dims}
\end{table}
\end{center}
\vspace{-7mm}
We know that $d(4)$ can be at most $e(4)=34$, and indeed the space of quartic relations
in $\kappa_1, \ldots, \kappa_4$ is one-dimensional, spanned by the Kummer surface equation.

In genus~3, Stubbs has found the following quadratic relation between the $\kappa_i$ and
shown that for a generic curve it is the unique quadratic relation up to scalars.
\begin{equation}\label{R1}
R_1:\kappa_1\kappa_8-\kappa_2\kappa_7-\kappa_3\kappa_6-\kappa_4\kappa_5-2f_5\kappa_2\kappa_4+f_5\kappa_3^2+2f_6\kappa_3\kappa_4+3f_7\kappa_4^2=0
\end{equation}
The dimensions for genus~3 are presented in Table~\ref{g3dims}. The existence and
uniqueness of $R_1$ implies that $d(2)=35$, but since $e(2)=36$, this means that there is
an even function in $\mathcal{L}(4\Theta)$ not coming from a quadratic monomial in the
$\kappa_i$, which does not happen in genus~2. Accordingly, we can at this point only bound
$d(3)\le~112$ and $d(4)\le~260$ from above. 

It follows that in genus~3 there must be at least~$70=330-260$ quartic relations on the Kummer
variety. But $36$ of these are multiples of the quadratic relation $R_1$.
Moreover, there are only~8 cubic relations and they are all multiples of $R_1$.
Hence there must be at least~34 independent irreducible quartic relations.
 \begin{table}\begin{tabular}{|c|c|c|c|}
\hline
 $n$ &$m(n)$ & $e(n)$ & $d(n)$  \\
\hline
1 &8&8&8 \\
2 &36&36&35 \\
3 &120&112&112 \\
4 &330&260&260 \\
\hline
\end{tabular}\caption{Dimensions in genus~3}\label{g3dims}\end{table}
In~\cite[Chapter~5]{Stubbs} Stubbs lists~26 quartic relations and
conjectures that together with $R_1$ these relations are independent and form a basis of the space of all
relations on the Kummer variety. His relations are at most quadratic in
$\kappa_5,\ldots,\kappa_8$. 
Using
current computing facilities we can verify the former conjecture quite easily, but because
of our dimension counting argument, we know that the latter conjecture cannot hold. 

To compute a complete set of defining equations for the Kummer variety we employ the technique already used by Stubbs. 
Because of the enormous
size of the algebra involved in these computations, simply searching for relations
among all monomials is not feasible.
Instead we split the monomials into parts of equal $x$-weight and
$y$-weight. These are homogeneous weights discussed in~\cite[\S3.5]{Stubbs} that were
already used by Flynn in~\cite{Flynna} in the genus~2 situation.
See Table~\ref{xywt}.
\begin{table}[h!]\begin{tabular}{|c|c|c|}
\hline
 &$x$ &$y$  \\
\hline
$x_i$&1&0\\
$y_i$&0&1\\
$f_i$&$-i$&2\\
$\kappa_i,i\le4$&$i-1$&0\\
$\kappa_i,i>4$&$i-9$&$2$\\
\hline
\end{tabular}\caption{$x$- and $y$-weight}\label{xywt}\end{table}
\\On monomials of equal $x$- and $y$-weight we can use linear algebra to find relations; we
continue this process with increasing weights until we have found enough quartic relations
to generate a space of dimension~70. The difficulty of this process depends essentially on
the $y$-weight. 
We used {\tt Magma}~\cite{Magma} to find~34 relations $R_2,\ldots,R_{35}$ on $K$, of $y$-weight at
most~8, such that the following result holds:
\begin{lemma}\label{relsindep}
Let $C$ be defined over a field $k$ of characteristic~$\ne 2,3,5$. 
Then the space \[\{R_2,\ldots,R_{35}\}\cup\{\kappa_i\kappa_jR_1:1\le i\le j\le8\}\] has
dimension equal to~70.
\end{lemma}
\begin{proof}
It suffices to show that the matrix of coefficients of the relations expressed as linear
combinations of the quartic monomials in the~$\kappa_i$ has rank~70.
Using~{\tt Magma}, we find a $70\times 70$-minor of this matrix whose determinant is equal
to the integer $2^{46}\cdot3^4\cdot5^5$. 
\end{proof}
These relations, as well as code for the verification alluded to in the proof,
can be downloaded at~\cite{HP}.
One can also verify that when $\Char(k) \in \{2,3,5\}$, the relations generate a space of
dimension strictly less than~70.

In order to show that for a given curve $C$ over a field $k$ of characteristic
$\notin\{2,3,5\}$ the
relations $R_1,\ldots,R_{35}$ actually generate all relations on the associated Kummer
variety $K$, we need to show that we have $d(4)=260$ for $C$.
We were unable to verify this for an arbitrary curve $C$ since the necessary manipulations seem
infeasible with current computing facilities.
However, by finding a single curve $C$ for which the space of cubic (resp. quartic) homogeneous polynomials has
dimension exactly~112 (resp.~260), we can conclude that we have $d(3)=112$ and $d(4)=260$
generically.
For this one can use, for instance, the curve given by $Y^2=Z^8+X^7Z$.

Using Proposition~\ref{KumQuart} we have therefore proved the following:
\begin{thm}\label{g3rels}
  Let $C:Y^2=F(X,Z)$ be a generic hyperelliptic curve of genus~3 defined over a field of
  characteristic $\ne 2,3,5$ and having a rational Weierstrass point at $\infty$. 
Then the relations on the Kummer threefold are generated by the relations $R_1,\ldots,R_{35}$.
\end{thm}
In practice, it is rather easy to check for a given curve whether $d(4) = 260$
and we provide code for doing so at~\cite{HP}.
We have carried this out for more than~10.000 curves $C$, including degenerate curves (i.e.,
such that $\disc(F) = 0$). Note that the definition of
$d(n)$ makes sense for degenerate curves.
In all cases we have found $d(4) = 260$, so it stands to reason that this should be true in
general:
\begin{conj}\label{d4260}
  Let $k$ be a field of characteristic $\ne2,3,5$ and let $F(X,Z) \in k[X,Z]$ be a
  binary octic form such that $\deg_X(F) =7$. 
  Then we have $d(4) = 260$.
\end{conj}

\section{Remnants of the group law}\label{sec:remnants}
Keeping the notation of the previous section, 
we now investigate which remnants of the group law on $A$ can be exhibited on $K$.
Namely, we recall results of Duquesne, show that analogues of the biquadratic form
representing pseudoaddition on the Kummer surface cannot exist in our situation and
conjecture formulas for duplication on $K$.
In light of Theorem~\ref{g3rels}, we assume that $C$ is defined over a field $k$ of
characteristic coprime to~30.
Note, however, that in the present section we do not have to assume that
$d(4) =260$.

Let $T$ be a~2-torsion point on $A$. Duquesne~\cite[\S~III.2.1]{SDthesis} has found a matrix $W_T$ such that projectively the identity
\[
\kappa(P+T)=W_T\cdot\kappa(P)
\]
holds for all $P\in A$ if we view $\kappa(P)$ and $\kappa(P+T)$ as column vectors. 
Duquesne's method of
finding $W_T$ is analogous to the method employed by Flynn~\cite{Flynnb} in the genus~2 case, 
although there are a few additional technical difficulties. 
We also have that if $T\in A(k)[2]$, then $W_T$ is defined over $k$.

Now let $P,Q\in A$. Then in general $\kappa(P+Q)$ and $\kappa(P-Q)$ cannot be found from
$\kappa(P)$ and $\kappa(Q)$, but the unordered pair $\{\kappa(P+Q),\kappa(P-Q)\}$ can be
found. 
In other words, the map from $\mathrm{Sym}^2(K)$ to itself that maps
$\{\kappa(P),\kappa(Q)\}$ to $\{\kappa(P+Q),\kappa(P-Q)\}$ is well-defined.
In fact, in the analogous situation in genus~2 there are biquadratic forms $B_{ij}\in
k[x_1,\ldots,x_4;y_1,\ldots,y_4]_{2,2}$ with the following property:
If $x$ and $y$ are Kummer coordinates for $P$ and $Q$, respectively, then there are 
Kummer coordinates $w$ for $\kappa(P+Q)$ and $z$ for $\kappa(P-Q)$ such that
\begin{equation}\label{wzBxy}
 w\ast z=B(x,y)
\end{equation}
holds.
Here~\eqref{wzBxy} is an abbreviation for 
\begin{eqnarray*}
 B_{ij}(x,y)&=&w_i z_j + w_j z_i\text{   for   }i\ne j\\
 B_{ii}(x,y)&=&w_i z_i.
\end{eqnarray*}
The following result says that in general such biquadratic forms cannot exist in genus~3.
\begin{prop}\label{nobij}
Let $A$ be the Jacobian of a generic hyperelliptic curve $C$ of genus~3
with a $k$-rational Weierstrass point, given
by an equation ~\eqref{ceq}, and let $K$ be the Kummer variety associated to $A$.
Then there is no set of biquadratic forms $B_{ij}(x,y)$, where $1\le i,j\le 8$, satisfying the following: 
If $x$ and $y$ are sets of Kummer coordinates for $P,Q\in A$, respectively, then there are Kummer coordinates $w,z$ for $P+Q,P-Q$, respectively, such that~\eqref{wzBxy} holds.
\end{prop}
\begin{proof}
We can work geometrically, so we assume $k$ is algebraically closed and let $C$ be a
hyperelliptic curve of genus~3 defined over $k$, given by a septic model $Y^2=F(X,Z)$, with
Jacobian $A$. 
Let us fix Kummer coordinates $x(T)=(x(T)_1,\ldots,x(T)_8)$ for all $T\in A[2]$. 
 
For each $T\in A[2]$ we get a map
\[
 \pi_T: k[x_1,\ldots,x_8;y_1,\ldots,y_8] \To k[y_1,\ldots,y_8],
\]
given by evaluating the tuple $x=(x_1,\ldots,x_8)$ at $x(T)$. This induces a map
\[
 \pi_T:\frac{k[x_1,\ldots,x_8;y_1,\ldots,y_8]_{2,2}}{(R_1(x),R_1(y))}\To \frac{k[y_1,\ldots,y_8]_2}{(R_1(y))}.
\]
Suppose a set of
forms $B_{ij}(x,y)$, $1\le i,j \le 8$, as in the statement of the proposition does exist
and consider 
\begin{equation}\label{R1B}
R_1(B):=B_{18}-B_{27}-B_{36}-B_{45}-2f_5B_{24}+2f_5B_{33}+2f_6B_{34}+6f_7B_{44}.
\end{equation}
Denote by $\overline{R_1(B)}$ the image of $R_1(B)$ in $\frac{k[x_1,\ldots,x_8;y_1,\ldots,y_8]_{2,2}}{(R_1(x),R_1(y))}$.
For $T\in A[2]$, arbitrary $P\in A$ and a set of Kummer
coordinates $y$ for $P$, we have:
If $B(x(T),y)=w\ast z$, then $w$ and $z$ are both Kummer coordinates for $P+T=P-T$,
and thus, if $x(T)$ and $y$ are scaled suitably so that $z=w$, we must have $B_{ij}(x(T),y)=2z_iz_j$ for $1\le i\ne j\le8$ and
$B_{i,i}(x(T),y)=z_i^2$ for $i\in\{1,\ldots,8\}$.
As an element of $K_\A$, the tuple $z$ must satisfy~\eqref{R1} and hence this implies
\begin{equation}\label{r1b0}
 \pi_T(\overline{R_1(B)})=R_1(z)=0\quad\text{ for all }T\in A[2]. 
\end{equation}
We claim that $\overline{R_1(B)}$ itself vanishes. In order to show this, we fix $T\in A[2]$ and let 
\[
S(T)=\{s_1(T),\ldots,s_{36}(T)\}=\{x(T)_ix(T)_j:1\le i\le j\le 8\}.
\]
We also fix a representative
\[
 \sum^8_{j=1}\sum^8_{l=1} \lambda_{T,j,l}\cdot y_j\cdot y_l 
\]
of $\pi_T(\overline{R_1(B)})$, where 
\[
 \lambda_{T,j,l}=\sum^{36}_{m=1}\mu_{T,j,l,m}\cdot s_m(T)
\]
is linear in the $s_m(T)$ and we require that $\lambda_{T,1,8}=0$, which uniquely determines our representative.

From~\eqref{r1b0} we know that we must have
\[
  \lambda_{T,j,l}=0
\] for all $j,l$ and for all $T\in A[2]$ and thus we get~64 linear equations
\[\sum^{36}_m\mu_{T,j,l,m}\cdot s_m(T)=0.\]

For notational purposes, denote the elements of $A[2]$ by $\{T_1,\ldots,T_{64}\}$.
It can be shown that the matrix $(s_i(T_j))_{1\le i\le 36,1\le j\le 64}$ has generic rank
equal to~35, by showing that this holds for a particular curve $C$, for instance for
$C:Y^2=X(X-Z)(X-2Z)(X-3Z)(X-4Z)(X-5Z)(X-6Z)$.
As usual, code for this computation can be found at~\cite{HP}.
It follows that for a generic curve $C$, any linear relation between the $s_i(T)$ satisfied by all 
$T\in A[2]$ must be a multiple of $R_1(x(T)_1,\ldots,x(T)_8)$. 

Suppose now that $C$ is generic, in the sense that  $(s_i(T_j))_{1\le i\le 36,1\le j\le
64}$ has rank~35.
Then $\overline{R_1(B})$ must vanish.
The upshot of this is that if we require our $B_{ij}(x,y)$ to contain no multiples of,
say, $x_1x_8$ or $y_1y_8$ as summands (which we can always arrange by applying
\eqref{R1}), then $R_1(B)=0$ follows. 

Now let $P\in A$  such that $2P$ is generic in the sense of
Section~\ref{sec:embedding} and let $x$ be a set of Kummer coordinates for $P$. 
Note that since $A$ is complete, such a point $P$ always exists.
Because any set of Kummer coordinates $z$ for $P-P =0$ satisfies $z_i =0$ for
$i=1,\ldots,7$ and $z_8\ne0$, we get $B_{ij}(x,x) = 0$ for $1\le i,j \le 7$
from~\eqref{wzBxy}.
However, $R_1(B) =0$ implies that also $B_{18}(x,x)=0$.
But for a set $w$ of Kummer coordinates for $2P$ we have $w_1 = B_{18}(x,x)$ up to a
nonzero rational factor.
This contradicts the assumption that $2P$ is generic in the sense of
Section~\ref{sec:embedding}, since for such points the first Kummer coordinate can never
vanish.
\end{proof}
\begin{rk}
  The assertion rk$(s_i(T_j))_{1\le i\le 36,1\le j\le 64} = 35$ can be verified easily
  for any given curve; we have tested about~1.000.000 examples and found that the rank
  was indeed~35 in all cases.
  However, we were unable to show this for arbitrary $C$. 
  For instance, it proved
  impossible to even write down a single $35\times35$-minor over the function field.
\end{rk}
Propostion~\ref{nobij} implies that the situation is rather more complicated than in genus~2. 
Recall Flynn's strategy to compute the biquadratic forms in genus~2 (see~\cite{Flynnb} or
\cite{CasselsFlynn}): If  $T\in A[2]$ and $P\in A$ is arbitrary, then we can compute 
\[
\kappa_i(P+T)\kappa_j(P-T)+\kappa_j(P+T)\kappa_i(P-T)=2\kappa_i(P+T)\kappa_j(P+T)
\]
projectively for all $i$ and $j$ by multiplying the matrix $W_T$ by the vector $\kappa(P)\in
k^4$. Using some algebraic manipulations, Flynn ensures that the resulting forms $B'_{ij}$
are biquadratic in the $\kappa_i(P)$ and the $\kappa_j(T)$ and satisfy some additional
normalization conditions. One can then check that the space generated by all $\kappa_i(T)\kappa_j(T)$, 
where $i\le j$, has dimension~10. Hence for each pair $(i,j)$ at most one biquadratic form that 
satisfies the same normalization conditions can specialize to $B'_{ij}$. 
The crucial point is that from classical theory of theta functions we already know that
biquadratic forms $B_{ij}$ satisfying~\eqref{wzBxy} must exist -- at least in the complex case (see Hudson's book~\cite{Hudson}) and thus, using the Lefshetz principle, for any algebraically closed field of characteristic~0.  Therefore Flynn concludes that $B_{ij}=B'_{ij}$ for all $i,j$.
 
We can try to use the same strategy in the genus~3 case. Indeed, in~\cite[\S~III.2.2]{SDthesis},
Duquesne computes the correct $B'_{ij}(x,y)$ in the special case that $x$ is a set of
Kummer coordinates for $T\in A[2]$. They can be downloaded from \url{ftp://megrez.math.u-bordeaux.fr/pub/duquesne}. Because of
the relation~\eqref{R1}, we know that the space of all $\kappa_i(T)\kappa_j(T)$, where
$i\le j$, is not linearly independent. But we also know that it has dimension~35, since $R_1$ is the only quadratic relation up to a constant factor. 
Now we can apply $R_1(x)$ and $R_1(y)$ to the $B'_{ij}(x,y)$ to make sure that no terms containing, say $x_1x_8$ or $y_1y_8$ appear and this is done by Duquesne. Thus we can draw the same conclusion as in the genus~2 situation, namely that for each pair $(i,j)$ at most one biquadratic form that satisfies the same normalization conditions can specialize to $B'_{ij}$. 
By Proposition~\ref{nobij},  we know that there is no set of biquadratic forms on $K$
satisfying~\eqref{wzBxy} in general. 
But we conjecture that we can still make use of the $B'_{ij}$ as follows:

We define two index sets 
\[
     I:=\{(i,j):1\le i\le j\le8\},
\]
and
\[
     E:=\{(1,8),(2,7),(3,6),(4,5),(5,5),(5,6),(5,7),(6,6)\}\subset I.
\]
We say that a pair of points $(P,Q)\in A\times A$ is {\em good} if there is a pair
$(i_0,j_0)\in I\setminus E$ such that if $x$ and $y$ are Kummer coordinates for $P$ and
$Q$, respectively, and $w$ and $z$ denote Kummer coordinates for $P+Q$ and $P-Q$,
respectively, then we have
\begin{enumerate}[(i)]
\item $B'_{i_0j_0}(x,y)\ne0$;
\item $w_{i_0}\ne 0$;
\item $z_{j_0}\ne 0$.
\end{enumerate}
If $(P,Q)$ is a good pair and $x,y,w,z$ are as above, then we can normalize $w$ and $z$ such 
that $w_{i_0}z_{j_0}=B'_{i_0j_0}(x,y)$. 
For $1\le i,j\le 8$ we define  $\alpha_{i,j}(x,y)$ as follows:
\begin{equation}\label{aij}
\alpha_{ij}(x,y):=w_iz_j+w_jz_i-B'_{ij}(x,y).
\end{equation}

Building on a large number of numerical experiments we state a list of conjectures regarding the relations between $B'_{ij}(x,y)$ and $w_iz_j+w_jz_i$:
\begin{conj}\label{kg3conj}
Suppose that $(P,Q)\in A\times A$ is a good pair with respective Kummer coordinates $x$
and $y$.
Then the following properties are satisfied:
\begin{itemize}
\item[(a)] We have $\alpha_{ij}(x,y)=0$ for $(i,j)\in I\setminus E$.
\item[(b)] The identities
\[
 -\alpha_{1,8}(x,y)=\alpha_{2,7}(x,y)=\alpha_{3,6}(x,y)=\alpha_{4,5}(x,y)
 \]
 and
 \[
 \alpha_{5,7}(x,y)=-2\alpha_{6,6}(x,y)
 \]
 hold. 
 \item[(c)] If $\alpha_{i_1j_1}(x,y)=0$ for some $(i_1,j_1)\in E$, then all $\alpha_{ij}(x,y)$ vanish.
 \item[(d)]
 If $\alpha_{i_1j_1}(x,y)\ne0$ for some $(i_1,j_1)\in E$, then we have
$\alpha_{ij}(x,y)\ne0$ for all $(i,j)\in E$. If this holds and if $(i,j),(i'j')\in E$, then the ratios 
 \[
 \frac{\alpha_{i'j'}(x,y)}{\alpha_{ij}(x,y)},  
 \]
only depend on $C$ and on $(i,j),(i',j')$, but not on $x$ or $y$. 
\end{itemize}
\end{conj}
\begin{rk}
The values $\alpha_{ij}(x,y)$ depend on the choice of the pair $(i_0,j_0)\in I\setminus
E$, but note that the assertions of Conjecture~\ref{kg3conj} are independent of this
choice.
\end{rk}
The naive height $h$ on the Kummer surface associated to a Jacobian surface $A$ can be used to
define and compute a canonical height $\hat{h}$ on $A$, which has several applications. 
See Flynn-Smart~\cite{FlynnSmart} for the construction and an algorithm for the computation of
$\hat{h}$, and~\cite{StollH2} for improvements due to Stoll.
For this application, one does not have to work with the biquadratic forms $B_{ij}$, but rather with the quartic duplication polynomials $\delta$ which, however, were originally derived from the $B_{ij}$. 
If we assume the validity of the first two parts of Conjecture~\ref{kg3conj}, then we can
find analogs of these polynomials in the genus~3 situation which again turn out to be
quartic.

More precisely,  we define
\[
    \delta'_i(x):=B'_{i8}(x,x)\in k[x]\quad\text{ for }i=2,\ldots,8,
\]
and 
\[
    \delta'_1(x):=\frac{4B'_{18}(x,x)+R(x)}{3}\in k[x],
\]
where $R(x)$ is a certain quartic relation on $K$ which we use to get rid of the
denominators in $\frac{4}{3}B_{18}(x,x)$.
Let $\delta'(x):=(\delta'_1(x),\ldots,\delta'_8(x))$.
We take the $\delta'_i$ as our candidates for the duplication polynomials on $K$.

As in the genus~2 situation, we want that the set $(0,\ldots,0,1)$ of Kummer coordinates
of the origin is mapped to itself by the duplication map.
This is required by the canonical height algorithms in~\cite{FlynnSmart} and
\cite{StollH2}.
In our situation we have
\[
 \delta'(0,0,0,0,0,0,0,1)=(0,0,0,0,0,0,0,f_7^2).
\]
But this can be fixed easily by a simple change of models of $K$
using the map
    \[
\tau(x_1,\ldots,x_8)=(x_1,\ldots,x_7,f_7x_8).
    \]
Setting
    \[
        \delta := \frac{1}{f^2_7}(\tau\circ\delta'\circ\tau^{-1})
    \]
    we find
\begin{itemize}
\item[1)] $\delta_i\in\Z[f_0,\ldots,f_7][x_1,\ldots,x_8]$ for all $i=1,\ldots,8$;
\item[2)] $\delta(0,0,0,0,0,0,0,1)=(0,0,0,0,0,0,0,1)$.
\end{itemize}

Based on extensive numerical evidence, we conjectured the following result in an earlier
version of this article (and in~\cite{MuellerThesis}):
\begin{thm}\label{deltaconj}
If $P\in A$, then
\[
\delta'(\kappa(P))=\kappa(2P)
\]
and
\[
\delta(\tau(\kappa(P)))=\tau(\kappa(2P)).
\]
\end{thm}
\begin{proof}
  In~\cite{StollKummer}, Stoll constructs polynomials which represent duplication on the image of $\xi$.
  It suffices to show that the map on $K$ obtained from his duplication map using the
transformation in Remark~\ref{StollTrans} coincides with the map $\delta'$, at least modulo the
relations $R_1,\ldots,R_{35}$.
To this end, we have used {\tt Magma}, see~\cite{HP}.
\end{proof}



\begin{thebibliography}{999}
{\frenchspacing
\footnotesize
 
  \bibitem{CAV}
 C. Birkenhake, H. Lange, {\em Complex Abelian Varieties}, 2nd edition, Springer-Verlag, Berlin (2004). 
  
  \bibitem{CasselsFlynn}
    J.W.S. Cassels and E.V. Flynn, {\em Prolegomena to a middlebrow
    arithmetic of curves of genus $2$}, Cambridge University Press, Cambridge (1996).
    \bibitem{SDthesis}
	  S. Duquesne, {\em Calculs effectifs des points entier et rationnels sur les
courbes}, The\`se de doctorat, Universit\'e Bordeaux~I (2001).
	\bibitem{SD}
	  S. Duquesne, {\em Traces of the group law on the Kummer surface of a curve of genus 2 in characteristic 2}, Preprint (2007).
  \bibitem{Flynna}
   E.V. Flynn, {\em The jacobian and formal group of a curve of genus 2 over an arbitrary ground field}, 
    {Math. Proc. Camb. Phil. Soc.} {\bf 107}, 425--441 (1990).

	\bibitem{Flynnb}
	   E.V. Flynn, {\em The group law on the jacobian of a curve of genus 2}, 
	    {J. reine angew. Math.} {\bf 439}, 45--69 (1993).
  \bibitem{FlynnSmart}
    E.V. Flynn and N.P. Smart, {\em Canonical heights on the Jacobians
    of curves of genus $2$ and the infinite descent}, 
    Acta Arith. {\bf 79}, 333--352 (1997).
\bibitem{holmes:height1}
    D. Holmes, {\em Computing N\'eron-Tate heights of points on hyperelliptic Jacobians},
     J. Number Theory {\bf 132}, 2, 1295--1305 (2012).
\bibitem{holmes:height2}
    D. Holmes, {\em An Arakelov-Theoretic Approach to Naive Heights on Hyperelliptic
Jacobians}, Preprint (2012). arXiv:math/1207.5948v2 [math.NT]

 \bibitem{Hudson}
	  R.W.H.T. Hudson, \emph{Kummer's Quartic Surface}, University Press, Cambridge (1905).

\bibitem{HP} Kummer threefold code,
        \url{http://www.uni-oldenburg.de/fileadmin/user_upload/mathe/personen/steffen.mueller/g3kummer.zip}.
 \bibitem{Magma}
   MAGMA is described in W. Bosma, J. Cannon and C. Playoust, 
  {\em The Magma algebra system I: The user language}, 
  J. Symb. Comp. {\bf 24}, 235--265 (1997). (See also the Magma home page at \url{http://magma.maths.usyd.edu.au/magma/}.)

 \bibitem{Kchar2}
  J.S. M\"uller, {\em Explicit Kummer surface formulas for arbitrary characteristic}, LMS J. Comput. Math. {\bf 13}, 47--64 (2010).
	\bibitem{MuellerThesis}
		J.S. M\"uller, {\em Computing canonical heights on Jacobians}, PhD thesis, Universit\"at Bayreuth (2010).

	\bibitem{mueller:computing}
		J.S. M\"uller, {\em Computing canonical heights using arithmetic intersection theory},
                  Math. Comp. {\bf 83}, 311--336 (2014). 
	\bibitem{MuellerStoll}
		J.S. M\"uller and M. Stoll, {\em Canonical heights on Jacobians of genus
                two curves}, in preparation (2014). (See also \url{http://www.mathe2.uni-bayreuth.de/stoll/talks/Luminy2012.pdf}.)

\bibitem{mumford:equations1}
D. Mumford, {\em    On the equations defining abelian varieties. {I}}, Invent. Math.
{\bf 1}, 287--354 (1966).

\bibitem{mumford:av}
D. Mumford, {\em Abelian varieties}, Tata Institute of Fundamental Research Studies in
Mathematics {\bf 5}, Tata Institute of Fundamental Research, Bombay (1974).
        
    \bibitem{StollH1}
    M. Stoll, {\em On the height constant for curves of genus two}, 
    Acta Arith. {\bf 90}, 183--201 (1999).
    
  \bibitem{StollH2}
    M. Stoll, {\em On the height constant for curves of genus two, II},
    Acta Arith. {\bf 104}, 165--182 (2002).

    \bibitem{StollKummer}
    M. Stoll, {\em  An explicit theory of heights for hyperelliptic Jacobians of genus
three}, in preparation (2014). (See also \url{http://www.mathe2.uni-bayreuth.de/stoll/talks/Luminy2012.pdf}.)

  \bibitem{Stubbs}
 A.G.J. Stubbs, {\em Hyperelliptic curves}, PhD thesis, University of Liverpool (2000).
  }

\end{thebibliography}
\end{document}